\documentclass [12pt] {article}

\usepackage{amsthm}
\theoremstyle {definition}
\newtheorem {definition} {Definition} [section]
\newtheorem {theorem} [definition] {Theorem}
\newtheorem {proposition} [definition] {Proposition}
\newtheorem {lemma} [definition] {Lemma}

\newtheorem {fact} [definition] {Fact}
\newtheorem {example} [definition] {Example}
\newtheorem {remark} [definition] {Remark}

\usepackage{natbib}

\usepackage{amsmath}
\usepackage{amsfonts}
\newcommand {\liff} {\leftrightarrow}

\newcommand {\Spec} {\mathrm {Spec}}
\newcommand {\I} {\mathcal I}

\newcommand {\comp}[1]{#1^{\mathrm c}}

\title{More on a Curious Nucleus}
\author{Levon Haykazyan}

\begin{document}

\maketitle

\begin{abstract}
Harold Simmons in [``A curious nucleus,'' {\em Journal of Pure and Applied
Algebra}, 2010] introduced a pre-nucleus and its associated nucleus that measure
the subfitness of a frame. Here we continue the study of this pre-nucleus.  We
answer the questions posed by Simmons.
\end{abstract}

\section{Introduction}

A {\em frame} $A$ is a complete lattice (our lattices will always be bounded)
satisfying the distributivity law
$$a \land \bigvee B = \bigvee \{a \land b : b \in B\},$$
for arbitrary $a \in A$ and $B \subseteq A$. Given a topological space $S$, the
lattice $\mathcal OS$ of open subsets of $S$ is a frame. The study of frames as
an algebraic approach to topology has been initiated by Ehresmann in 1950s.
Frames later acquired independent existence as a generalisation of topology that
is applicable to wider settings.

In an arbitrary distributive lattice $D$ we can consider the preorder
$$a \preceq b \text { iff } (\forall c \in D) [a \lor c = \top \implies b \lor c
= \top]$$
that refines the order $\le$ of the lattice.

A frame is called {\em subfit} if $\preceq$ agrees with $\le$. Subfitness was
introduces by \cite {isbell-locales} to show that compact subfit frames are
spatial (isomorphic to the lattice of open sets of a topological space). It was
rediscovered by \cite{simmons-separation} as a weak separation property.

Prompted by \cite{coquand-compact}, Simmons proposed in \cite{simmons-curious}
to study the map $\xi : A \to A$
$$\xi(a) = \bigvee \{b \in A : b \preceq a\}$$
for an arbitrary frame $A$. In this way, the frame $A$ is subfit if and only if
$\xi$ is the identity.

As any algebraic structure, a frame can be studied through its quotients. There
is an elegant way of handling quotients of a frame: a quotient is completely
determined by the map taking each element to the join of all elements with the
same image. Such maps are called {\em nuclei}. We can compare two nuclei (or
more generally arbitrary maps) $f, g : A \to A$ pointwise, i.e. $f \le g$ iff
for every $a \in A$ we have $f(a) \le g(a)$. This makes the set of all nuclei
into a poset, which turns out to be a frame too. Meets in this frame are
computed pointwise, however describing joins is more complicated. For this
reason, one often looks at a larger class of maps called {\em pre-nuclei} (see
the definition below). The map $\xi$ is a pre-nucleus.

\cite{simmons-curious} poses a number of questions about the pre-nucleus $\xi$
and its relationship with other nuclei. We answer these questions in this paper.

In section \ref{section-two} we study the pre-nucleus on a general frame.  We
show that $\xi$ is equal to the join of all (pre-)nuclei that only admit the top
element $\top$ (Theorem \ref{F_n=xi-theorem}). Next we construct an example to
show that the pre-nucleus $\xi$ (and consequently the join of all (pre-)nuclei
admitting only the top element) need not be a nucleus (Example
\ref{no-nucleus-example}). To obtain a nucleus from $\xi$, in general, we
therefore need to iterate it. We show that there is no a priori bound on the
closure ordinal of $\xi$ (Theorem \ref{no-bound-theorem}).

In section \ref{section-three} we study $\xi$ on the frame $\I D$ of ideals of a
distributive lattice $D$ where $\xi$ is idempotent and hence a nucleus. (This is
the setting in \cite{coquand-compact} which motivated the definition of $\xi$ in
\cite {simmons-curious}. However the nucleus $\xi$ in this setting has already
appeared in \cite{johnstone-amax}. This reference was apparently missed by both
\cite{coquand-compact} and \cite{simmons-curious}.) Actually $\I D$ is the
frame of open subsets of the spectrum $\Spec D$ of $D$. So a nucleus picks out
the subspace of prime ideals fixed by it. The subspace determined by $\xi$ is
characterised in \cite{johnstone-amax} as the soberification of the
space of maximal ideals of $D$. This characterisation immediately implies some
equivalent conditions for $\I D$ to be subfit (Proposition
\ref{jacobson-equivalence}). We call such lattices Jacobson.

Following \cite{simmons-curious} we compare $\xi$ to another nucleus $\chi : \I
D \to \I D$ defined by $\chi(I) = \bigcup_{c \in I} \{a \in D : a \preceq c\}$
(note that $\preceq$ here is computed in $D$ rather than $\I D$). Actually
$\preceq$ preserves (binary) meets and joins and therefore determines a
congruence $\equiv$ on $D$. From this it follows that the subspace of $\Spec D$
determined by $\chi$ is homeomorphic to $\Spec (D/{\equiv})$. We characterise it
as the least spectral subspace of $\Spec D$ containing maximal ideals (Theorem
\ref{chi-characterisation-theorem}). Lastly we show that $\chi = \xi$ if and
only if $D/{\equiv}$ is a Jacobson lattice (Theorem \ref{chi=xi-theorem}).

\subsection* {Acknowledgement}

The author is grateful to James Parson and the anonymous referee for careful
reading, comments and suggestions for improving the paper.

\section{The Pre-Nucleus $\xi$}
\label{section-two}

We assume some familiarity with the theory of frames. The necessary background
can be found e.g. in \cite{johnstone-stone-spaces} or \cite{picado-pultr-book}.
Let $A$ be a frame.

\begin{definition}
\begin{itemize}
\item An {\em inflator} is a map $f : A \to A$ that is inflationary and
monotone, i.e.
$$a \le f(a) \text { and } a \le b \implies f(a) \le f(b)$$
for all $a, b \in A$.
\item A {\em pre-nucleus} is an inflator that satisfies
$$f(a) \land f(b) \le f(a \land b)$$
for all $a, b \in A$. (The inequality is actually an equality by monotonicity.)
\item A {\em nucleus} is an idempotent pre-nucleus, i.e. it satisfies
$$f^2 = f \circ f = f.$$
\end{itemize}
\end{definition}

Before proceeding further, let us stress that in some literature (e.g. in
\cite{picado-pultr-book} and in \cite{banaschewski-tychonoff}, where it is
introduced) the term pre-nucleus is used for an inflator that
satisfies the weaker property $f(a) \land b \le f(a \land b)$. We follow
\cite{simmons-curious} in our choice of terminology.

Given any family $\mathcal F$ of functions from $A$ to $A$ we can consider their
pointwise join $\bigvee \mathcal F : A \to A$ defined by
$$(\bigvee \mathcal F)(a) = \bigvee \{f(a) : f \in \mathcal F\}.$$
The pointwise join of a nonempty directed family of pre-nuclei is a pre-nucleus.

For a pre-nucleus $f : A \to A$ we say that $f$ admits $a \in A$ if $f(a) =
\top$.  Let $\mathcal F_n$ and $\mathcal F_p$ be the set of all nuclei and
pre-nuclei respectively that only admit $\top$: that is the set of all
respective $f$ such that $f(a) = \top \implies a = \top$.  Since $\mathcal F_n
\subseteq \mathcal F_p$ we have that $\bigvee \mathcal F_n \le \bigvee \mathcal
F_p$. Note that $\mathcal F_p$ is closed under composition and is therefore
directed. It follows that $\bigvee \mathcal F_p$ is a pre-nucleus. We will later
see that, in contrast, $\bigvee \mathcal F_n$ need not be a nucleus.

Now let $f \in \mathcal F_p$ and $a, b \in A$. Since $f$ is an inflator we have
that $f(a) \lor b \le f(a \lor b)$. Therefore we have
$$f(a) \lor b = \top \implies f(a \lor b) = \top \implies a \lor b = \top.$$
We conclude that $f(a) \preceq a$ and therefore $f \le \xi$. It follows that
$\bigvee \mathcal F_p \le \xi$. We now show that it is in fact an equality,
answering question (5) of \cite{simmons-curious}.

\begin{theorem}
\label{F_n=xi-theorem}
For an arbitrary frame $A$ we have $\xi \le \bigvee \mathcal F_n$. Therefore
$\bigvee \mathcal F_n = \bigvee \mathcal F_p = \xi$.
\end{theorem}
\begin{proof}
Fix $a \in A$ and let $b \preceq a$. We construct a nucleus $g \in \mathcal F_n$
such that $b \le g(a)$. Let 
$$G(c) = \{x \in A : x \le b \lor c \text { and } x \land a \le c\}$$
and define $g(c) = \bigvee G(c)$. Since $c \in G(c)$ we conclude that $c \le
g(c)$. Note also that $c \le c' \implies G(c) \subseteq G(c') \implies g(c) \le
g(c')$. Therefore $g$ is an inflator.

Also by frame distributivity, $G(c)$ is closed under arbitrary joins and
in particular $g(c) \in G(c)$. Now given $c, c' \in A$ we have $g(c) \land g(c')
\le (b \lor c) \land (b \lor c') = b \lor (c \land c')$ and $(g(c) \land g(c'))
\land a \le c \land c'$. Hence $g(c) \land g(c') \in G(c \land c')$ and
therefore $g(c) \land g(c') \le g(c \land c')$. Thus $g$ is a pre-nucleus.

The map $g$ is actually a nucleus. To see this we need to use the fact that
$g(c) \in G(c)$. So let $d \in G(g(c))$.  Then $d \le b \lor g(c) \le b \lor (b
\lor c) = b \lor c$. Also $d \land a \le g(c)$. But we also have $d \land a \le
a$. Therefore $d \land a \le g(c) \land a \le c$. We conclude that $d \in G(c)$,
which implies that $g(g(c)) = g(c)$.

We claim that $g \in \mathcal F_n$. Indeed assume that $g(c) = \top$. Then $\top
\in G(c)$ and therefore $\top \le b \lor c$ and $\top \land a \le c$.
Which is to say $b \lor c = \top$ and $a \le c$. But since $b \preceq a$, the
equality $b \lor c = \top$ implies that $a \lor c = \top$ and therefore $c =
\top$.

Finally we have $b \le b \lor a$ and $b \land a \le a$. Therefore $b \in G(a)$
and so $b \le g(a) \le (\bigvee \mathcal F_n)(a)$. Thus $\xi(a) = \bigvee \{b
\in A : b \preceq a\} \le (\bigvee \mathcal F_n)(a)$.
\end{proof}

Next we give an example of a frame to show that $\xi$ (and therefore $\bigvee
\mathcal F_p$) need not be a nucleus. This answers the question (1) (and also
question (3)) of \cite{simmons-curious}. Our example is actually spatial (i.e.
the frame of opens of a topological space).

Given a topological space $S$ consider the frame $\mathcal O S$ of open subsets.
In this case for $W, U \in \mathcal O S$ we have
$$U \preceq W \text { iff } \forall u \in (U \setminus W)[u^- \not \subseteq
U],$$
where $u^-$ denotes the topological closure of $\{u\}$ (see
\cite{simmons-curious}).

\begin{example}
\label{no-nucleus-example}
Let $\omega = \{0, 1, ...\}$ be the first infinite ordinal and $\omega^+ =
\omega \cup \{*\}$, where $*$ is a new element. The universe of our topological
space is $\omega \times \omega^+$. The basis of the topology is given by sets of
the following form
$$U \times \{y\},$$
where $U \subseteq \omega$ is downward closed and $y \in \omega$ and
$$\omega \times V \cup U \times \{*\},$$
where $U \subseteq \omega$ is downward closed and $V \subseteq \omega$ is
cofinite.

We claim that $\xi(\emptyset) = \omega \times \omega$. Indeed given $(x,y) \in
\omega \times \omega$ consider its open neighbourhood $U = [0,x+1) \times
\{y\}$. The point $(x+1, y)$, which is outside of $U$ is in the closure of every
point of $U$. This shows that $U \preceq \emptyset$ and therefore $(x,y) \in
\xi(\emptyset)$. Now consider a point $(x, *)$. Each basic open neighbourhood
$W$ of it contains a subset of the form $\omega \times \{y\}$, with $y \in
\omega$, which is closed. Therefore $W \not \preceq \emptyset$.

To see that $\xi$ is not a nucleus note that $\xi^2(\emptyset) = \xi(\omega
\times \omega) = \omega \times \omega^+$. Indeed given a point $(x, *)$,
consider its open neighbourhood $U = \omega \times \omega \cup [0, x+1) \times
\{*\}$. Now the point $(x+1, *)$, which is not in $U$, is in the closure of
every point of $U \setminus \omega \times \omega$. This shows that $U \preceq
\omega \times \omega$.
\end{example}

So in general $\xi$ is not a nucleus. However for any pre-nucleus there is a
least nucleus above it. To obtain it we need to iterate the pre-nucleus as
follows. For an ordinal $\alpha$ define $\xi^\alpha$ by
\begin{itemize}
\item $\xi^0$ is the identity function;
\item $\xi^{\alpha+1} = \xi \circ \xi^\alpha$;
\item $\xi^\alpha = \bigvee_{\beta < \alpha} \xi^\beta$ if $\alpha$ is a limit
ordinal.
\end{itemize}
Then for each frame $A$ (by cardinality considerations) there is an ordinal
$\alpha$ such that $\xi^{\alpha+1} = \xi^\alpha$, which will then be a nucleus.
This ordinal $\alpha$ in general depends on $A$. We next show that there is no a
priori bound, answering question (2) (and question (4)) of
\cite{simmons-curious}.

For that we need to analyse the above example in more details. We recall two
constructions on frames.

The first construction is the Cartesian product (or just product) of frames.
Suppose we have a collection $(A_i : i \in I)$ of frames. Their Cartesian
product is the frame with underlying set $\Pi_{i \in I} A$ where the operations
of meet and join are performed coordinatewise. We denote the Cartesian product
of two frames $A_1$ and $A_2$ by $A_1 \times A_2$. It follows that $(a_i)_{i \in
I} \preceq (b_i)_{i \in I}$ if and only if $a_i \preceq b_i$ for every $i \in
I$. Therefore $\xi((a_i)_{i \in I}) = (\xi(a_i))_{i \in I}$ and more generally
$\xi^\alpha((a_i)_{i \in I}) = (\xi^\alpha(a_i))_{i \in I}$. Note that there is
a slight abuse of notation here, since we use the same letter $\xi$ for nuclei
on all frames.

The second construction is the tensor product (or coproduct) of frames. We only
need binary tensor products, so we introduce only binary ones. Let $A$ and $B$
be frames. Their tensor product $A \otimes B$ is the frame presented by
generators $a \otimes b$, where $a \in A$ and $b \in B$, subject to the
following relations
\begin{itemize}
\item $(a_1 \otimes b_1) \land (a_2 \otimes b_2) = (a_1 \land a_2) \otimes (b_1
\land b_2)$, so meets are computed coordinatewise;
\item $\bigvee_i (a_i \otimes b) = (\bigvee_i a_i) \otimes b$, in particular for
a nullary join we have $\bot = \bot \otimes b$;
\item $\bigvee_i (a \otimes b_i) = a \otimes (\bigvee_i b_i)$, and thus $\bot =
a \otimes \bot$.
\end{itemize}
The general element of $A \otimes B$ has the form $\bigvee_i a_i \otimes b_i$,
though this representation is not unique.

Now let $A$ and $B$ be frames and let $\mathcal P(\omega)$ denote the frame of
all subsets of $\omega$. Denote $C = B \otimes \mathcal P(\omega)$ and consider
the subframe $D$ of $C \times A$ such that $(c, a) \in D$ iff
either $a = \bot_A$ or there is a cofinite subset $Y \subseteq \omega$ such that
$\top_B \otimes Y \le c$. It is easy to check that $D$ is closed under finite
meets and arbitrary joins. For notational purposes below we write $\xi(c, a)$
for $\xi((c,a))$ and $\comp Y$ for the complement $\omega \setminus Y$ of a
subset $Y \subseteq \omega$.

\begin{lemma}
In the frame $D$ the following hold:
\begin{enumerate}
\item if $\top_B \neq b \in B$, then $\xi(b \otimes \omega, \bot_A) = (\xi(b)
\otimes \omega, \bot_A)$;
\item $\xi(\top_C, a) = (\top_C, \xi(a))$.
\end{enumerate}
\end{lemma}

\begin{proof}
\begin{enumerate}
\item Let $(c, a) \preceq (b \otimes \omega, \bot_A)$ in $D$. First we show that
$a = \bot_A$. Assume otherwise; then by the definition of $D$, there is a
cofinite subset $Y \subseteq \omega$ such that $\top_B \otimes Y \le c$. Let $j
\in Y$. Consider the element $(\top_B \otimes \comp{\{j\}},
\top_A)$. It's join with $(c, a)$ is $\top_D$, however the join with $(b
\otimes \omega, \bot_A)$ is not (since $b \neq \top_B$). This is a
contradiction showing that $a = \bot_A$.

Now let $c = \bigvee_i b_i \otimes X_i$ and $X_i \neq \emptyset$. We claim that
$(c, \bot_A) \preceq (b \otimes \omega, \bot_A)$ implies that $b_i \preceq b$.
It is enough to assume that $c = b_0 \otimes X_0$ and $X_0 \neq \emptyset$. We
need to show that $b_0 \preceq b$. Let $b' \in B$ be such that $b_0 \lor b' =
\top_B$. Pick $j \in X_0$ and consider 
$$((\top_B \otimes \comp{\{j\}}) \lor (b' \otimes \{j\}), \top_A) \in D.$$
Its join with $(c, \bot_A)$ is $\top_D$ and therefore so is its join with $(b
\otimes \omega, \bot_A)$. But this implies that	$b \lor b' = \top_B$, which
proves that $b_0 \preceq b$. From this we deduce that $\xi(b \otimes \omega,
\bot_A) \le (\xi(b) \otimes \omega, \bot_A)$.

To see the converse, we show that if $b_0 \preceq b$, then $(b_0 \otimes \omega,
\bot_A) \preceq (b \otimes \omega, \bot_A)$. Let $(c', a') \in D$ be such that
$((b_0 \otimes \omega) \lor c', a') = \top_D$. Then $a' = \top_A$ and therefore
there is a cofinite subset $Y = \omega \setminus \{i_0, ..., i_{n-1}\}$ such
that $\top_B \otimes Y \le c'$. Thus $c' = \top_B \otimes Y \lor \bigvee_{j =
0}^{n-1} b_{i_j} \otimes \{i_j\}$ for some $b_{i_j} \in B$. It follows that for
each $j = 0, ..., n-1$ we have $b_{i_j} \lor b_0 = \top_B$. Thus by the
assumption that $b_0 \preceq b$ we get that $b_{i_j} \lor b = \top_B$. From this
we deduce that $((b \otimes \omega) \lor c', a') = \top_D$.

\item This follows from the following easy observation $(c', a') \preceq
(\top_C, a)$ if and only if $a' \preceq a$. 
\end{enumerate}
\end{proof}

\begin{theorem}
\label{no-bound-theorem}
For every ordinal $\alpha > 0$ there is a frame $A_\alpha$ such that $\alpha$ is
the least ordinal with $\xi^\alpha(\bot_A) = \top_A$.
\end{theorem}
\begin{proof}
By induction on $\alpha$.
\begin{itemize}
\item For $\alpha = 1$ we can pick e.g. $A_1$ to be the frame of downward closed
subsets of $\omega$.
\item Assume the hypothesis for $\alpha$. Let $A_{\alpha+1}$ be the subframe of
$(A_\alpha \otimes \mathcal P(\omega)) \times A_1$ described in the previous
lemma. Then, by that lemma, $\alpha$ is the least ordinal such that
$$\xi^\alpha(\bot_{A_{\alpha+1}}) = (\xi^\alpha(\bot_{A_\alpha}) \otimes \omega,
\bot_{A_1}) = (\top_{A_\alpha} \otimes \omega, \bot_{A_1}).$$
Hence 
$$\xi^{\alpha+1}(\bot_{A_{\alpha+1}}) = (\top_{A_\alpha} \otimes \omega,
\xi(\bot_{A_1})) = \top_{A_{\alpha+1}}$$
and $\alpha+1$ is the least such.
\item Let $\alpha$ be a limit ordinal and assume the hypothesis for all $\beta <
\alpha$. Take $A_\alpha = \prod_{\beta < \alpha} A_\beta$.
\end{itemize}
\end{proof}

\section {The Nuclei $\xi$ and $\chi$ on a Frame of Ideals}
\label{section-three}

Let $D$ be a distributive lattice. Then the lattice $\I D$ of all ideals of $D$
is a frame. It is the frame of open sets of the spectrum $\Spec D$ of $D$.
Recall that $\Spec D$ is the space of prime ideals of $D$ in the topology
generated by basic open sets of the form $\{I \in \Spec D: a \not \in I\}$ for
$a \in D$. Since we have two lattices now, we will use lowercase letters for
elements of $D$ and uppercase letters for elements of $\I D$.

Let $I \in \I D$ and $a \in D$. Denote by ${\downarrow} a$ the principal ideal
$\{b \in D: b \le a\}$. We then have ${\downarrow} a \preceq I$ iff for every
ideal $J \in \I D$ such that ${\downarrow} a \lor J = D$ we have $I \lor J = D$.
It is enough to check for principal ideals $J = {\downarrow} b$. Thus we can
reformulate this as ${\downarrow} a \preceq I$ iff for every $b \in D$ we have $a
\lor b = \top_D$ implies there is $c \in I$ such that $c \lor b = \top_D$.
Finally observing that $a \in \xi(I)$ iff ${\downarrow} a \preceq I$ we get
$$a \in \xi(I) \iff (\forall b \in D) \big[a \lor b = \top_D \implies (\exists c
\in I) [c \lor b = \top_D]\big].$$

From this characterisation it is easy to check that $\xi^2(I) = \xi(I)$ and
therefore $\xi$ is a nucleus on $\I D$. The nucleus $\xi$ picks out the subspace
$(\Spec D)_\xi = \{I \in \Spec D : \xi(I) = I\}$ of $\Spec D$ whose frame of
opens is the lattice $(\I D)_\xi = \{I \in \I D: \xi(I) = I\}$. Part of question
(7) of \cite{simmons-curious} asks to characterise the space $(\Spec D)_\xi$. In
fact such a characterisation had already appeared in \cite{johnstone-amax}.

\begin{fact} [\cite{johnstone-amax}]
\label{jacobson-fact}
For every ideal $I \in \I D$, the ideal $\xi(I)$ is the intersection of maximal
ideals containing $I$. The space $(\Spec D)_\xi$ is the soberification of the
subspace of maximal ideals of $D$. It is compact and Jacobson. (Recall that a
space is Jacobson if every closed set is the closure of its closed points).
\end{fact}

We have the following corollary.

\begin{proposition}
\label{jacobson-equivalence}
The following conditions are equivalent on a distributive lattice $D$.
\begin{enumerate}
\item Every prime ideal is an intersection of maximal ideals.
\item Every ideal is an intersection of maximal ideals.
\item The frame $\I D$ is subfit, i.e. $\xi$ is the identity on $\I D$.
\item The spectrum $\Spec D$ is a Jacobson space.
\end{enumerate}
\end{proposition}
\begin{proof}
$1 \implies 2$ holds since every ideal of a distributive lattice is the
intersection of prime ideals containing it.

$2 \implies 3$. Assuming $2$, by Fact \ref{jacobson-fact} we get that the
nucleus $\xi$ is the identity on $\I D$. Therefore $\I D$ is subfit.

$3 \implies 4$ again follows from Fact \ref{jacobson-fact}.

$4 \implies 1$. Assume that $\Spec D$ is a Jacobson space. Let $I \in \Spec D$
and consider its closure $I^{-}$. By the assumption, it is the closure of its
closed points $\{J_l : l \in L\}$. Then each $J_l$ is a maximal ideal containing
$I$. We claim that $I = \bigcap_{l \in L} J_l$. Indeed suppose that $a \in D
\setminus I$. Then $I$ is in the basic open set $D(a) = \{J \in \Spec D: a \not
\in J\}$. Since $I$ is in the closure of $\{J_l : l \in L\}$ there is an $l \in
L$ such that $J_l \in D(a)$. But then $a \not \in J_l$.
\end{proof}

We are not aware of studies of such lattices in the literature. However, given
the obvious analogy with rings, there is only one name we can give to them.

\begin{definition}
A distributive lattice satisfying the equivalent conditions of Proposition
\ref{jacobson-equivalence} is called a {\em Jacobson lattice}.
\end{definition}

\cite{simmons-curious} compares $\xi$ to another nucleus $\chi : \I D \to \I D$
defined by
$$\chi(I) = \bigcup_{c \in I} \{a \in D: a \preceq c\}.$$
The map $\chi$ is indeed a nucleus and it is easy to check that $\chi \le \xi$.
In the rest of this section we will characterise the space $(\Spec D)_\chi = \{I
\in \Spec D: \chi(I) = I\}$ as well as characterise the lattices $D$ for which
$\chi = \xi$.

For this we need to look at the preorder $\preceq$ in more details.  We would
like to make an order out of $\preceq$. For that we need to quotient $D$ by the
equivalence relation $\equiv$ defined by $a \equiv b$ if $a \preceq b$ and $b
\preceq a$.

\begin{lemma}
\label{top-lemma}
The relation $\equiv$ is the largest lattice congruence on $D$ such that $a
\equiv \top$ implies $a = \top$.
\end{lemma}
\begin{proof}
It is easy to check that $\equiv$ is a congruence. If $a \equiv \top$, then
$\top \preceq a$. But since $\top \lor \bot = \top$ the definition of $\preceq$
gives $a = a \lor \bot = \top$.

Now suppose $\equiv'$ is a congruence with this property. Let $a, b \in D$ be
such that $a \equiv' b$. Fix $c \in D$ with property $a \lor c = \top$. Then we
have $\top \equiv' a \lor c \equiv' b \lor c$ and hence $b \lor c = \top$. This
shows that $a \preceq b$. Since $\equiv'$ is symmetric we also get $b \preceq a$
and so $a \equiv b$.
\end{proof}

We now consider the lattice $D/{\equiv}$ (with ordering $\preceq$) and the
associated surjection from $D$. In general, if $D'$ is another distributive
lattice and $f : D \to D'$ a homomorphism, then the inverse image map $f^{-1}
:\Spec D' \to \Spec D$ is a spectral map (preimages of compact open sets are
compact open). Conversely any spectral map from a spectral space to $\Spec D$
induces a lattice homomorphism.  This is the well known Stone representation
theorem for distributive lattices. The surjective lattice homomorphisms are
characterised as follows.

\begin{fact} [\cite{dst-spectral}]
\label{surjective-lemma}
A lattice homomorphism $f : D \to D'$ is surjective if and only if the spectral
map $f^{-1} : \Spec D' \to \Spec D$ is a homeomorphism onto its image. 
\end{fact}

Thus lattice congruences of $D$ are in one-to-one correspondence between
spectral subspaces of $\Spec D$ in the following sense.

\begin{definition}
Let $X$ be a spectral space and $Y \subseteq X$ be a subspace that is also
spectral. We say that $Y$ is a {\em spectral subspace} of $X$ if the inclusion
map is spectral. In other words if $Z \subseteq X$ is compact and open, then $Z
\cap Y$ is also compact.
\end{definition}

\begin{remark}
\label{congruence-image-remark}
Let $\equiv'$ be a congruence on $D$, then the image of the canonical map
$\Spec D/{\equiv'} \to \Spec D$ can be characterised as those prime ideals $I$
of $D$ that are stable under $\equiv'$, i.e. $a \in I$ and $a \equiv' b$ implies
$b \in I$. Indeed, preimages of ideals of $D/{\equiv'}$ are clearly stable under
$\equiv'$. Conversely any ideal $I \in \Spec D$ that is stable under $\equiv'$
is the preimage of the ideal $\{a/{\equiv'} : a \in I\} \in \Spec D/{\equiv'}$.
\end{remark}

Now we can characterise the space $(\Spec D)_\chi$. This, together with Fact
\ref{jacobson-fact}, answers question (7) of \cite{simmons-curious}.

\begin{theorem}
\label{chi-characterisation-theorem}
The space $(\Spec D)_\chi$ is the least spectral subspace of $\Spec D$
containing the set of maximal ideals.
\end{theorem}

\begin{proof}
The space $(\Spec D)_\chi$ is a spectral subspace by Lemma
\ref{surjective-lemma}. Further since $\chi$ fixes all maximal ideals, they are
contained in $(\Spec D)_\chi$.

Now let $X \subseteq \Spec D$ be a spectral space containing all maximal
ideals. Let $D'$ be the lattice of its compact open sets. Then by Lemma
\ref{surjective-lemma}, there is a surjective homomorphism $f : D \to D'$. This
map $f$ induces a congruence $a \equiv' b$ iff $f(a) = f(b)$ on $D$. Note that
if $a \neq \top$, then there is a maximal ideal $I$ of $D$ containing $a$. By
the assumption, $I \in X$ and so is stably under $\equiv'$. It follows that $a
\not \equiv' \top$. Thus by Lemma \ref{top-lemma} $\equiv' \subseteq \equiv$ and
so all ideals that are stable under $\equiv$ are stable under
$\equiv'$. Hence by Remark \ref{congruence-image-remark} $(\Spec D)_\chi
\subseteq X$.
\end{proof}

\begin{remark}
James Parson has pointed out to me that the quotient lattice $D/\equiv$ has been
studied in \cite{clq-heitmann} under the name {\em Heitmann lattice of $D$}. The
above characterisation of $(\Spec D)_\chi$ also appears there.
\end{remark}

We can now answer question (6) of \cite{simmons-curious} by characterising the
lattices for which $\xi$ and $\chi$ agree.

\begin{theorem}
\label{chi=xi-theorem}
Let $D$ be a distributive lattice. Then $\chi = \xi$ if and only if $D/{\equiv}$
is a Jacobson lattice.
\end{theorem}

\begin{proof}
Assume that $\chi = \xi$. Then $\Spec (D/{\equiv})$ is homeomorphic to
$(\Spec D)_\chi = (\Spec D)_\xi$ which is Jacobson by Fact \ref{jacobson-fact}.
Hence $D/{\equiv}$ is a Jacobson lattice.

Conversely assume that $D/{\equiv}$ is a Jacobson lattice. Let $I \in \I D$ be
an ideal such that $I = \chi(I)$. We show that $\xi(I) = I$. Let $a \not \in I$.
It is enough to find a maximal ideal extending $I$ and not containing $a$.
Consider the ideal $I/{\equiv}$ of $D/{\equiv}$. Since the latter is a Jacobson
lattice, there is a maximal ideal $J'$ of $D/{\equiv}$ extending $I/{\equiv}$
and not containing $a/{\equiv}$. Let $J = \{b \in D : b/{\equiv} \in J'\}$ be
the preimage of $J'$ under the canonical map. Then $I \subseteq J$ and $a \not
\in J$. But if $c \not \in J$, then there is $b \in J$ such that $b \lor c
\equiv \top_D$ (since $J' = J/{\equiv}$ is maximal). Then by Lemma
\ref{top-lemma} we get that $b \lor c = \top_D$ showing that $J$ is a maximal
ideal.
\end{proof}

As a final remark we highlight some connections with model theory.  The passage
from $D$ to $D/{\equiv}$ (and therefore the nucleus $\chi$) has a manifestation
in model theory: it corresponds to the passage from a theory to its largest
companion. This can most elegantly be expressed in coherent logic (also called
positive model theory). However it can already be seen in the classical Robinson
model theory, which is the setting we adopt here. Suppose $T$ is an inductive
theory (i.e. $\forall \exists$-axiomatisable) and $\Pi_n$ is the set of
universal formulas in $n$ variables $\bar x = (x_1, ..., x_n)$ (we identify two
formulas $\phi(\bar x)$ and $\psi(\bar x)$ if $T \models \forall \bar x
(\phi(\bar x) \liff \psi(\bar x))$). Then $\Pi_n$ is naturally a distributive
lattice with the order $\phi(\bar x) \le \psi(\bar x)$ iff $T \models \forall
\bar x (\phi(\bar x) \to \psi(\bar x))$. Then it can be shown that $\phi \preceq
\psi$ in $\Pi_n$ if and only if $\forall \bar x (\psi(\bar x) \to \phi(\bar x))$
is true in every existentially closed model of $T$.

Let $T_{\mathrm K}$ denote the $\forall \exists$-theory of all existentially
closed models of $T$.  Then $T_{\mathrm K}$ and $T$ are companions meaning that
they imply the same $\Pi_n$ formulas. The fact that $T_{\mathrm K}$ is the
largest companion of $T$ is precisely because of Lemma \ref{top-lemma}. We can
further see that $T$ will have a model companion (i.e. every model of
$T_{\mathrm K}$ is existentially closed) if and only if for every $n$ the
lattice $\Pi_n/{\equiv}$ is a boolean algebra (and therefore also Jacobson).

\bibliographystyle{plainnat}
\bibliography{../all}

\end{document}